\documentclass[11pt,a4paper]{article}

\usepackage[OT1]{fontenc}
\usepackage[utf8]{inputenc}

\usepackage[english]{babel}
\usepackage{amsfonts}
\usepackage{amsthm}
\usepackage{mathtools}  

\newtheorem{teo}{Theorem}[section]

\newtheorem{lem}[teo]{Lemma}

\newtheorem{rem}[teo]{Remark}

\def\a{\mathcal A}
\def\wa{\widetilde{\a}}

\begin{document}

\title{\vspace*{-2cm}\sc The case of equality in  H\"older's  inequality for matrices and operators}





\date{}
\author{G. Larotonda\footnote{Instituto Argentino de Matem\'atica (CONICET), Universidad Nacional de General Sarmiento, Universidad de Buenos Aires. e-mail: glaroton@dm.uba.ar}}

\maketitle

\setlength{\parindent}{0cm} 

\vspace*{-1cm}\begin{abstract}
Let $p>1$ and $1/p+1/q=1$. Consider H\"older's inequality 
$$
\|ab^*\|_1\le \|a\|_p\|b\|_q
$$
for the $p$-norms of some trace ($a,b$ are matrices, compact operators, elements of a finite $C^*$-algebra or a semi-finite von Neumann algebra). This note contains a simple proof (based on the case $p=2$) of the fact that equality holds iff $|a|^p=\lambda |b|^q$ for some $\lambda\ge 0$.
\end{abstract}

\section{Introduction}\label{intro1}

The purpose of this note is to give a geometrical (and simple) proof of the fact that equality holds in H\"older's inequality for the $p$-norms of matrices or operators $a,b$ if and only if $|a|^p=\lambda|b|^q$ for a precise $\lambda\ge 0$. 

\medskip

A first proof of this result for $a,b\ge 0$ goes back to Dixmier \cite{dix} where he first shows that $a,b$ should commute and then via the spectral theorem he reduces the problem to the equality in the classical H\"older inequality. Yet another different proof is based on the $s$-numbers of operators and majorization theory, i.e. the proof given by M. Manjegani in \cite{m}; that proof depends on the solution of the case of equality in Young's inequality  for nuclear operators which was given in \cite{af} (for the case of equality for the singular values of compact operators, see  \cite{lar}).

\medskip

Let $\a=M_n(\mathbb C)$ or with more generality, any semi-finite von Neumann algebra with semi-finite faithful normal trace $\tau$. In our discussion, we  include $C^*$-algebras with a finite trace, because they can be embedded into its double commutant which is a finite von Neumann algebra by a classical result of Takesaki \cite[Proposition V.3.19]{t}. Moreover, in the semi-finite case, the argument works without modification for unbounded $a,b\in \wa$, the algebra of $\tau$-measurable operators affiliated with $\a$ (see  Nelson's paper \cite{n}). 

\subsection{Notation and the Cauchy-Schwarz inequality}

We will denote the $p$-norms with $\|x\|_p=(\tau|x|^p)^{1/p}$ for $p\ge 1$, and  $\|x\|$ will denote the norm of $x$ in the algebra $\a$. What follows is the well-known H\"older inequality for $a,b\in\a$:
$$
\|ab^*\|_1\le \|a\|_p\|b\|_q.
$$
Note that in particular $\tau(|a|^p) ,\tau(|b|^q)<\infty$ implies $|\tau(ab)|\le \|ab\|_1<\infty$. For a proof of this inequality for $a,b\in M_n(\mathbb C)$, see Bhatia's book \cite{bhatia}. For compact operators the standard reference is the book of Simon \cite{simon}, and for the continuous case, see Nelson's paper \cite{n} on noncommutative integration.

\medskip

Since the main result of this note is based on it, let us start by recalling the well-known Cauchy-Schwarz inequality  with precision; for a proof see Proposition 2.1.3 in Kadison and Ringrose's book \cite{kr}.

\begin{lem}\label{cs}
Let $x,y\in \a$ and set $\langle x,y\rangle =\tau(xy^*)$, $\|\cdot\|_2=\sqrt{\langle x,x\rangle}$. Then 
 $$
| \langle x,y\rangle| \le \|x\|_2\|y\|_2.
 $$
Moreover, if $\tau(xy^*)=\|x\|_2\|y\|_2$, then $x=\lambda y$ for some $\lambda\ge 0$.
\end{lem}

\begin{rem}\label{posi}
Let $a=u|a|$ be the polar decomposition of $a\in\a$, then $u^*u|a|=|a|$. Write $a=xy^*$ with $x=u|a|^{1/2}$ and $y=|a|^{1/2}$, and use Cauchy-Schwarz to obtain $|\tau(a)|\le \tau(|a|)$ with (finite) equality $\tau(a)=\tau|a|=1$ if and only if $a=|a|\ge 0$.
\end{rem}

\section{H\"older}

We are now ready to consider the case of equality in H\"older's inequality.

\begin{teo}\label{eqholder}
If $p>1$ and $a,b\ne 0$ are in $\a$, equality holds in H\"older inequality
$$
\|ab^*\|_1=\|a\|_p\|b\|_q<\infty
$$
if and only if $\displaystyle\frac{|a|^p}{\|a\|_p^p}=\frac{|b|^q}{\|b\|_q^q}$.
\end{teo}
\begin{proof}
Let $b=\nu|b|$ be the polar decomposition of $b$, then $|ab^*|=\nu||a||b||\nu^*$ and $\nu^*|ab^*|\nu=||a||b||$ therefore $\|ab^*\|_1=\||a||b|\|_1$ for all $a,b$. We will write $p_a$ for the projection onto the range of $a$. Using the homogeneity, it suffices to consider the case $\|a\|_p=\|b\|_q=1$. 

If $|a|^p=|b|^q$, then $\|a\|_p^p=\|b\|_q^q$, $\|a\|_p\|b\|_q=\|a\|_p^p$ and $|a||b|=|a|^p$. Hence
$$
\|ab^*\|_1=\||a||b|\|_1=\||a|^p\|_1= \|a\|_p^p=\|a\|_p\|b\|_q.
$$
To prove the converse, write $||a||b||=w^*|a||b|$. Without loss of generality we can assume that $p\in (1,2]$ (if $p=2$, then $|a|^0$ denotes $p_a$ and this proof can be considerably shortened). Then
$$
1=\|ab^*\|_1=\||a||b|\|_1=\tau(w^*|a||b|)=\tau(w^*|a|^{p/2} \, |a|^{1-p/2}|b|)=\tau(xy^*)
$$
with $x=w^*|a|^{p/2}, y=|b||a|^{1-p/2}$, which by the Cauchy-Schwarz inequality is less or equal than
$$
\|x\|_2\|y\|_2\le \tau(|a|^p)^{1/2} \||b||a|^{1-p/2}\|_2= \||b||a|^{1-p/2}\|_2=\tau(|b|^2 |a|^{2-p})^{1/2}
$$
since $\|a\|_p=1$. Now, pick $r=q/2\ge 1$, $r'$ its conjugate exponent, then by H\"older's inequality, 
$$
\tau(|b|^2 |a|^{2-p})\le \||b|^2|a|^{2-p}\|_1\le (\tau|b|^q)^{1/r} (\tau (|a|^{(2-p)r'}))^{1/r'}=(\tau(|a|^{(2-p)r'}))^{1/r'}
$$
since $\|b\|_q=1$. But $(2-p)r'=p$, hence the expresssion is less or equal than $1$, thus all the expressions are equal. Now, note first that
$$
1=\tau(|b|^2 |a|^{2-p})= \||b|^2|a|^{2-p}\|_1 =\tau(||b|^2|a|^{2-p}|)
$$
and this is only possible (Remark \ref{posi}) if $|b|^2|a|^{2-p}\ge 0$, which can only happen if $|b|^2$ commutes with $|a|^{2-p}$, or equivalently, if $|a|$ commutes with $|b|$; in particular $w=1$ and $\tau(|a||b|)=\|ab^*\|_1=1$.

On the other hand we have also shown that $0\le \tau(xy^*)=\|x\|_2\|y\|_2$ and by  Lemma \ref{cs}, it follows that $x=\lambda y$ for some $\lambda\ge 0$, in our case
$|a|^{p/2}=\lambda |b||a|^{1-p/2}$ which implies $|a|^p=\lambda |b||a|$. Taking traces we get
$$
1=\|a\|_p^p=\lambda \tau(|a||b|)=\lambda \|ab^*\|_1=\lambda.
$$
Then $\lambda=1$ and we can also assert that $|a|^{p-1}=|b|p_a$. But then $|b|^q p_a=p_a|b|^q=|a|^{q(p-1)}=|a|^p$ and
$$
\tau(p_a|b|^q)=\tau(|a|^p)=1=\tau(|b|^q),
$$
or equivalently $\tau((1-p_a)|b|^q)=0$, which is only possible if $|b|^q=p_a|b|^q=|a|^p$ by the faithfulness of the trace.
\end{proof}

\begin{rem}
For the case of $p=1$, assume $\|ab^*\|_1=\|a\|_1\|b\|_\infty$. If one goes through the previous proof (take $r=\infty$, $r'=1$), arrives to $p_a\|b\|_\infty=p_a|b|=|b|p_a$, which is the necessary and sufficent condition to obtain the equality just mentioned (note that $p_a$ is then, in the $L^2(M,\tau)$ representation of $\a$, a norming eigenvector of $|b|$).
\end{rem}

\subsection*{Acknowledgements}

This research was supported by Instituto Argentino de Matem\'atica ''Alberto P. Calder\'on'', Universidad Nacional de General Sarmiento, Universidad de Buenos Aires, CONICET (PIP 2010-0757) and ANPCyT (PICT 2010-2478).

\end{document}